\documentclass[11pt,reqno]{amsart}
\usepackage{amsmath,amssymb,amsthm, hyperref, amscd}
\usepackage[T1]{fontenc}
\usepackage{textcomp}
\usepackage[scale=0.95,ttdefault]{AnonymousPro}
\hypersetup{
    colorlinks=true,
    linkcolor=blue,
    urlcolor=cyan,
    }
\usepackage[noadjust]{cite}
\renewcommand{\citepunct}{;\penalty\citemidpenalty\ }

\title[Relative $F$-rational signature]{The gamma-construction and permanence properties of the (relative) $F$-rational signature}

\newcommand{\fa}{\mathfrak{a}}

\newcommand{\fm}{\mathfrak{m}}
\newcommand{\fn}{\mathfrak{n}}

\newcommand{\fp}{\mathfrak{p}}

\newcommand{\Spec}{\operatorname{Spec}}
\newcommand{\rank}{\operatorname{rank}}

\newcommand{\HT}{\operatorname{ht}}

\newcommand{\cC}{\mathcal{C}}
\newcommand{\cS}{\mathcal{S}}

\newcommand{\bA}{\mathbf{A}}
\newcommand{\bF}{\mathbf{F}}

\newcommand{\bQ}{\mathbf{Q}}
\newcommand{\bP}{\mathbf{P}}

\newcommand{\bZ}{\mathbf{Z}}

\newcommand{\Tag}[1]{\href{https://stacks.math.columbia.edu/tag/#1}{\texttt{#1}}}
\newcommand{\citestacks}[1]{\cite[Tag \Tag{#1}]{Stacks}}
\newcommand{\citetwostacks}[2]{\cite[Tags \Tag{#1} and \Tag{#2}]{Stacks}}

\newcommand{\ppower}[2]{{#1}^{[p^{#2}]}}
\newcommand{\ehk}[1]{e_{\mathrm{HK}}({#1})}

\author{Shiji Lyu}
\address{Department of Mathematics, Statistics, and Computer Science\\University of Illinois at Chicago\\Chicago, IL
60607-7045\\USA}
\email{\href{mailto:slyu@uic.edu}{slyu@uic.edu}}
\urladdr{\url{https://homepages.math.uic.edu/~slyu/}}

\newtheorem{Thm}{Theorem}[section]

\newtheorem{Lem}[Thm]{Lemma}
\newtheorem{Cor}[Thm]{Corollary}
\newtheorem{Prop}[Thm]{Proposition}

\theoremstyle{definition}
\newtheorem{Def}[Thm]{Definition}
\newtheorem{Constr}[Thm]{Construction}

\theoremstyle{remark}
\newtheorem{Rem}[Thm]{Remark}

\begin{document}
\begin{abstract}
   We study some permanence properties of the relative $F$-rational signature defined and studied by Smirnov--Tucker.
   We show that this invariant is compatible with the gamma-construction,
   and then derive other main results from the $F$-finite case established by Smirnov--Tucker.
   These results can be seen as quantitative versions of results on $F$-rationality of V\'elez.
   We also obtain limited results about the $F$-rational signature defined and studied by Hochster--Yao.
   We explore some features of the gamma-construction along the way,
   which may be of independent interest.
\end{abstract}
\maketitle

\noindent\textsc{Conventions}. 
Throughout this paper,
$p$ is a fixed prime number. All rings contain $\bF_p$.
For an ideal $\fa$ of a ring $R$ and a positive integer $e$,
$\ppower{\fa}{e}$ denotes the ideal generated by $\{a^{p^e}\mid a\in\fa\}$,
or equivalently, the ideal generated by $\{a^{p^e}\mid a\in \Sigma\}$
where $\Sigma$ is a set of generators of $\fa$.
The length of a finite length module $M$ over a ring $R$ is denoted $l_R(M)$,
and usually just $l(M)$ if $R$ is clear from the context.
An \'etale-local map of local rings is a local map of local rings that is a localization of an \'etale ring map.
The completion of a local ring $A$ is denoted $A^\wedge$.

\section{Introduction}


For a Noetherian local ring $(R,\fm)$
and an $\fm$-primary ideal $I$,
the \emph{Hilbert--Kunz function} of $I$ is the fuction
$e\mapsto {l(R/\ppower{I}{e})}$.
It is studied by Kunz \cite{Kun76} (at least in the case $I=\fm$),
and the
\emph{Hilbert-Kunz multiplicity}
$e_{\mathrm{HK}}(R,I)=\lim_e \frac{l(R/\ppower{I}{e})}{p^{e\dim R}}$
is shown to exist by Monsky \cite{Mon83}.

The Hilbert--Kunz multiplicity has a direct connection to tight closure theory.
In a Noetherian reduced equidimensional complete local ring $(R,\fm)$, given two $\fm$-primary ideals $I\subseteq J$,
$J$ is contained in the tight closure of $I$ if and only if
$e_{\mathrm{HK}}(I)=e_{\mathrm{HK}}(J)$ \cite[Theorem 8.17]{HH90}.
Here we omitted the ring $R$ from the notations.
Hochster--Yao \cite{HY} and Smirnov--Tucker \cite{ST} defined the \emph{$F$-rational signature}
$s_{\mathrm{rat}}(R)$ and the \emph{relative $F$-rational signature} $s_{\mathrm{rel}}(R)$
of a Noetherian local ring $(R,\fm)$,
\begin{align*}
    s_{\mathrm{rat}}(R)&=\inf\{\ehk{I_0}-\ehk{I_0+uR}\mid I_0\text{ is a parameter ideal of }R, u\not\in I_0\}\\
    s_{\mathrm{rel}}(R)&=\inf\left\{\frac{\ehk{I_0}-\ehk{I}}{l(R/I_0)-l(R/I)}\middle\vert I_0\text{ is a parameter ideal of }R, I\supsetneq I_0\right\}.
\end{align*}
These invariants are known to characterize the $F$-rationality of $R$ under mild conditions
\citeleft\citen{HY}\citemid Theorem 4.1\citepunct \citen{ST}\citemid Proposition 2.1\citeright.
Various permanence properties for $s_\mathrm{rat}$ are studied in \cite{HY};
in \cite{ST}, via the alternative characterization of $s_\mathrm{rel}$ as the dual $F$-signature \cite[Corollary 5.9]{ST} in the $F$-finite case,
more permanence properties of $s_\mathrm{rel}$ are obtained.

In this paper, we continue the study in \cite{ST} 
on properties of the relative $F$-rational signature,
especially for non-$F$-finite rings.
The classical way of moving from the non-$F$-finite setting to the $F$-finite setting 
is via the $\Gamma$-construction of Hochster--Huneke, 
which we review in \S\S\ref{subsec:GammaBasic}.
Our main result is the following; for explicit statements, 
see Theorems \ref{thm:GammaSrel}, \ref{thm:RegExtn}, and \ref{thm:SemiContOnGring}.

\begin{Thm}
The relative $F$-rational signature is compatible with the $\Gamma$-construction,
ascends along a flat extension of local rings with geometrically regular closed fiber,
and defines a lower semi-continuous function on the spectrum of a finite type algebra over a local G-ring.\footnote{Readers unfamiliar of G-rings can refer to \citestacks{07GG} and \cite[(34.A)]{Matsumura}.}
\end{Thm}
Weaker results and difficulties in the case of the $F$-rational signature are discussed in \S\ref{sec:Srat}.

The following result is used for the proof of Theorem \ref{thm:RegExtn} and may be of independent interest.
For an explicit statement see Corollary \ref{cor:RegAGammaToBGamma}.
\begin{Thm}
A regular local map between Noetherian complete local rings can be approximated by regular local maps between $F$-finite Noetherian local rings.
\end{Thm}

The strategy is as follows.
Modulo Corollary \ref{cor:RegAGammaToBGamma},
which is proved in \S\ref{sec:GammaReg}, we can apply the $F$-finite case in \cite{ST} to prove our main theorem,
as soon as we prove compatibility with the $\Gamma$-construction.
To do so,
we use a uniform boundedness result of Polstra \cite{Pol} to replace the Hilbert-Kunz multiplicity $\ehk{R,I}$ 
with an individual value of the Hilbert-Kunz function $l(R/\ppower{I}{e})$ normalized by a factor of $p^{-e\dim R}$.
See \S\ref{sec:Unif},
where a slightly stronger version of Polstra's theorem is recorded.

For a fixed $e$,
we prove compatibility of the ``truncated'' invariant in \S\ref{sec:TruncatedGamma}.
The method is more or less model-theoretic;
ideals of our concern are parameterized by matrices,
and various lengths are virtually constructible functions of the entries of the matrices.
(We do not require any model-theoretic knowledge,
and the argument is carried out using the language of abstract and linear algebra.)
We are then facing the following question:
if a system of polynomial equations
defined over a field $k$ has a solution over $k^\Gamma$
for all small $\Gamma$,
does it necessarily have a solution over $k$?
The answer is yes, at least up to a finite separable extension of $k$,
see Lemma \ref{lem:AllGammaEpoint}.
This extra complication is settled in \S\ref{sec:fet}, owing to a trick
that only works for the relative $F$-rational signature and not the $F$-rational signature.\\

\textsc{Acknowledgements}. The author thanks Thomas Polstra, Kevin Tucker, Karl Schwede, and Ilya Smirnov for helpful discussions.

\section{Preparation}\label{sec:preliminary}
\subsection{\texorpdfstring{$\Gamma$}{Gamma}-construction}\label{subsec:GammaBasic}
We recall the classical $\Gamma$-construction of Hochster--Huneke \cite{HH94}.

\begin{Constr}[{\cite[(6.11)]{HH94}}]\label{constr:Gamma}
Let $A$ be a Noetherian complete local ring, $k$ a coefficient field of $A$.
Let $\Lambda$ be a $p$-basis of $k$
and $\Gamma\subseteq\Lambda$ a cofinite subset.

For a positive integer $e$, let $k^{\Gamma,e}$ be the field obtained from adjoining the $p^{e}$th root of all elements in $\Gamma$,
and let
$k^\Gamma=\bigcup_e k^{\Gamma,e}$
and $A^\Gamma=\bigcup_e k^{\Gamma,e}\hat{\otimes}_k A$.

Finally, for a finite type $A$-algebra $R$ and $Q\in\Spec(R)$, let $R^\Gamma=A^\Gamma\otimes_A R$
and $Q^\Gamma=\sqrt{QR^\Gamma}$.
\end{Constr}
%

We shall use the following basic facts about the $\Gamma$-construction.

\begin{Lem}[Hochster--Huneke]\label{lem:GammaFacts}
Let $A$ be a Noetherian complete local ring, $k$ a coefficient field of $A$.
Let $\Lambda$ be a $p$-basis of $k$.
%

Let $R$ be a finite type $A$-algebra.
Then the followings hold.

\begin{enumerate}
    \item\label{GammaNoeFfin} For all cofinite subset $\Gamma$ of $\Lambda$,
    $R^\Gamma$ is Noetherian and $F$-finite, hence excellent;
    the map $R\to R^\Gamma$ is flat with Gorenstein fibers.

    \item\label{PiGammaUnivHomeo} For all cofinite subsets $\Gamma$ of $\Lambda$,
    $\Spec(R^\Gamma)\to \Spec(R)$ is a universal homeomorphism, thus for all $Q\in\Spec(R)$,
    $Q^\Gamma$ is a prime ideal.
    
    \item\label{QRGammaPrime} For every $Q\in\Spec(R)$, there exists a cofinite subset $\Gamma_0$ of $\Lambda$,
    such that for all cofinite subsets $\Gamma$ of $\Gamma_0$,
    $Q^\Gamma=QR^\Gamma$.
    
    \item\label{FoverKQLinearDisjt} For every $Q\in\Spec(R)$ and every finite extension $F/\kappa(Q)$ of fields,
    there exists a cofinite subset $\Gamma_0$ of $\Lambda$,
    such that for all cofinite subsets $\Gamma$ of $\Gamma_0$,
    $\kappa(Q^\Gamma)$ is linearly disjoint with $F$ over $\kappa(Q)$.
\end{enumerate}
\end{Lem}
\begin{proof}
For (\ref{GammaNoeFfin})(\ref{PiGammaUnivHomeo}), see {\cite[(6.11)]{HH94}}.
(\ref{QRGammaPrime}) is part of \cite[Lemma 6.13(b)]{HH94}.
(\ref{FoverKQLinearDisjt}) follows from the proof of \cite[Lemma 6.13(b)]{HH94}, but let us formally deduce it as follows.
Let $S$ be a finite $(R/Q)$-algebra that is an integral domain with $F=\operatorname{Frac}(S)$.
From (\ref{PiGammaUnivHomeo})(\ref{QRGammaPrime}) we see that there exists a $\Gamma_0$ such that for all cofinite subsets $\Gamma$ of $\Gamma_0$, both $QR^\Gamma$ and $0\subseteq S^\Gamma$ are prime.
Then we see from (\ref{PiGammaUnivHomeo}) that
\[
F\otimes_{\kappa(Q)}\kappa(Q^\Gamma)=F\otimes_R R^\Gamma
=F\otimes_S S^\Gamma=\operatorname{Frac}(S^\Gamma)
\]
is a field,
as desired.
\end{proof}


\subsection{Solutions to polynomial equations}
We wish to show that if a system of polynomial equations has a solution over all sufficiently small $k^\Gamma$, 
then it has a solution over $k$.
We can only prove this after taking a finite separable extension of $k$, see below.
The result is not stated with the $\Gamma$-construction,
but in a more general form, in view of Lemma \ref{lem:GammaFacts}(\ref{FoverKQLinearDisjt}).

\begin{Lem}\label{lem:AllGammaEpoint}
Let $K/k$ be a field extension, $X$ a $k$-scheme of finite type,
and let $\{K_\alpha\}_\alpha$ be a cofiltered family of subextensions of $K/k$,
such that for every finite purely inseparable extension $F/k$,
there exists an $\alpha$ such that $F$ and $K_\alpha$ are linearly disjoint over $k$.

Assume that $X(K_\alpha)\neq\emptyset$ for all sufficiently small $K_\alpha$.
Then $X(E)\neq\emptyset$ for some finite separable extension $E/k$.
\end{Lem}
\begin{proof}
By Noetherian induction we may assume that $X$ is integral 
and that for all nonempty opens $U\subseteq X$,
$U(K_\alpha)\neq\emptyset$ for all sufficiently small $K_\alpha$.
We may therefore replace $X$ by any nonempty open subscheme,
and in particular we may assume $X$ regular (cf. \citestacks{07PJ}).

Let $F/k$ be finite purely inseparable such that $(X_F)_\mathrm{red}$ is geometrically reduced over $F$.
%
Take $\alpha$ such that $X(K_\alpha)\neq \emptyset$ and that $K_\alpha$ is linearly disjoint with $F$ over $k$.
Then $K_\alpha\otimes_k F$ is a field and we have a commutative diagram of schemes
\[
\begin{CD}
\Spec(K_\alpha\otimes_k F)@>>> X_F\\
@VVV @VVV\\
\Spec(K_\alpha)@>>> X
\end{CD}
\]
Let $\{x\}$ be the image of the bottom horizontal arrow.
Then the diagram shows that
$X_F\times_X \Spec(\kappa(x))\times_{\Spec(\kappa(x))}\Spec(K_\alpha)=\Spec(K_\alpha\otimes_k F)$
is the spectrum of a field,
thus $X_F\times_X \Spec(\kappa(x))$ is the spectrum of a field.
Since $X$ is regular, 
$X_F$ is regular at preimages of $x$,
and shrinking $X$
we may assume $X_F$ regular, in particular reduced.
Then $X_F$ is geometrically reduced over $F$, 
so $X$ is geometrically reduced over $k$.
But every geometrically reduced
$k$-scheme of finite type contains a nonempty smooth open subscheme, thus an $E$-point for some finite separable $E/k$.
\end{proof}
\begin{Rem}
    In general, it is not always true that $X(k)\neq\emptyset$.
    The following example was suggested by Sean Cotner.
    Let $k=\bF_p$ and let $K$ be the algebraic closure of $k$,
    then for any geometrically connected $X$ and any infinite subextension $K_\alpha$ of $K$ one has $X(K_\alpha)\neq\emptyset$,
    by the Weil conjectures.
    If we take, for example, $\{K_n\}_n$ to be the family of maximal extensions 
    coprime to $n$, where $n$ runs over square-free integers,
    then we see $X(K_n)\neq\emptyset$.
    If $X$ is taken so that $X(k)=\emptyset$,
    for example $X=\bP^1_k\setminus \bP^1_k(k)$,
    then we see the extension $E/k$ is necessary.

    However, one may expect that if $K/k$ is purely inseparable,
    then the extension $E/k$ is not necessary.
    We do not know the truth or falsehood of that statement.
\end{Rem}

\section{Uniform control}\label{sec:Unif}

In this section we define and begin our study on some invariants ``truncated'' from the definition of relative $F$-rational signature.

\begin{Def}\label{def:truncatedSrel}
Let $(R,\fm)$ be a Noetherian local ring, $I_0$ an $\fm$-primary ideal, and $e$ a positive integer.
For an ideal $I$ of $R$ properly containing $I_0$ we set
\[
    \lambda_{I_0}^e(R,I):=\frac{l(\ppower{I}{e}/\ppower{I_0}{e})}{p^{e\dim R}l(I/I_0)}
    =\frac{l(R/\ppower{I_0}{e})-l(R/\ppower{I}{e})}{p^{e\dim R}(l(R/I_0)-l(R/I))};
\]
and we set 
\begin{align*}
    \cS_{I_0}^{e}(R)&:=\{\lambda_{I_0}^e(R,I)\mid \fm I\subseteq I_0\subsetneq I\},\\
s_{I_0}^{e}(R)&:=\min \cS_{I_0}^{e}(R).
\end{align*}

Note that $\cS_{I_0}^{e}(R)$ is contained
in the finite set $[0,l(R/\ppower{I_0}{e})]\cap \frac{1}{p^{e\dim R}l(R/I_0)!}\bZ$,
so the minimum is attained.
\end{Def}

\begin{Lem}
\label{lem:FlatLengthPropotional}
Let $(R,\fm)\to (S,\fn)$ be a flat map of Noetherian local rings such that $\fm S$ is $\fn$-primary.
Let $I_0$ be an $\fm$-primary ideal.

Then for an ideal $I$ of $R$ properly containing $I_0$ and any positive integer $e$,
$IS$ is an ideal of $S$ properly containing $I_0S$,
and $\lambda_{I_0}^e(R,I)=\lambda_{I_0S}^e(S,IS)$.
Consequently, if $\fm S=\fn$, then 
$\cS_{I_0}^{e}(R)\subseteq \cS_{I_0S}^{e}(S)$ and $s_{I_0}^{e}(R)\geq s_{I_0S}^{e}(S)$;
and if $R/I_0\cong S/I_0S$,
then $\cS_{I_0}^{e}(R)= \cS_{I_0S}^{e}(S)$ and $s_{I_0}^{e}(R)= s_{I_0S}^{e}(S)$.
\end{Lem}
\begin{proof}
One observes that for all ideals $\fa$ of $R$ and all positive integers $e$, $\ppower{\fa}{e}S=\ppower{(\fa S)}{e}$ by definition.
Moreover, by flatness, $l_S(M\otimes_R S)=l_R(M)l_S(S/\fm S)$ for all $R$-modules $M$ of finite length.
Thus the lemma is clear.
\end{proof}

\begin{Def}\label{def:ConstantControl}
Let $R$ be a Noetherian ring, $P\in\Spec(R)$.
Let $C$ be a positive real number.

We say \emph{$C$ controls $F$-colength differences for $P\in\Spec(R)$} if for all $PR_P$-primary ideals $I_0\subsetneq I$, and all positive integers $e<e'$,
\[
|\lambda_{I_0}^e(R_P,I)-\lambda_{I_0}^{e'}(R_P,I)|\leq Cp^{-e},\label{diffEE'}
\]
and consequently
\[
\left| \lambda_{I_0}^e(R_P,I) -\frac{\ehk{I_0}-\ehk{I}}{l(R_P/I_0)-l(R_P/I)}\right|\leq Cp^{-e}.
\]
We say \emph{$C$ controls $F$-colength differences for $R$} if $C$ controls $F$-colength differences for all $P\in\Spec(R)$.
\end{Def}

The relation between our truncated invariants and the (relative) $F$-rational signature is the following.

\begin{Lem}\label{lem:truncatedandnot}
Let $(R,\fm)$ be a Noetherian local ring whose completion is equidimensional.\footnote{By a result of Ratliff, this is equivalent to that $R$ is equidimensional and universally catenary, see \citetwostacks{0AW4}{0AW6}.
In particular, if $R$ is equidimensional and a homomorphic image of a Cohen-Macaulay ring, then $R^\wedge$ is equidimensional, see \citestacks{00NM}.}

Let $I_0$ be a parameter ideal of $R$, and let $C$ be a real number that controls $F$-colength differences for $\fm\in\Spec(R)$.

Then for all positive integers $e$, we have
\[
    \left| s_{I_0}^e(R) -s_\mathrm{rel}(R)
\right|\leq Cp^{-e}.
\]
\end{Lem}
\begin{proof}
This follows from \cite[Corollary 3.7]{ST}.\footnote{The author belives that the mild assumption on $R$ that is necessary for \cite[Theorem 7.15]{HH94} to apply was overseen in the proof of \cite[Corollary 3.7]{ST}. 
One can also apply \cite[Theorem 4.2(c)(d)]{HH94} to get $R$ Cohen-Macaulay when the right hand side in \cite[Corollary 3.7]{ST} is not 0; 
this also requires formal equidimensionality.}
\end{proof}

\begin{Thm}[Polstra]\label{thm:Polstra}
The followings hold.
\begin{enumerate}
    \item\label{ExistsC} If $R$ is essentially of finite type over a Noetherian local ring, then there exists a $C$ that controls $F$-colength differences for $R$.
    
    \item\label{Cgoodforflatextn} If $R\to R'$ is a flat map of Noetherian rings, $C$ controls $F$-colength differences for some $P'\in \Spec(R')$, and $\HT P'=\HT(P'\cap R)$, then $C$ controls $F$-colength differences for $P'\cap R\in \Spec(R)$.
    
    \item\label{CgoodforflatextnGLOBAL}
    If $R\to R'$ is faithfully flat map of Noetherian rings, $C$ controls $F$-colength differences for $R'$, then $C$ controls $F$-colength differences for $R$.
\end{enumerate}
\end{Thm}
\begin{proof}
(\ref{Cgoodforflatextn}) follows from
Lemma  \ref{lem:FlatLengthPropotional}:
$\lambda_{I_0}^e(R_P,I)=\lambda_{I_0R'_{P'}}^e(R'_{P'},IR'_{P'})$ for all $I_0,I,e$,
thus if $C$ controls $F$-colength differences for $P'\in\Spec(R')$,
then it does for $P\in\Spec(R)$.

Now
(\ref{CgoodforflatextnGLOBAL}) follows from (\ref{Cgoodforflatextn}), since for all $P\in\Spec(R)$, a minimal prime $P'$ of $PR'$ is such that $P'\cap R=P$ and that $\HT P'=\HT P$.
Finally, (\ref{ExistsC}) follows from (\ref{CgoodforflatextnGLOBAL}) and \cite[Theorem 3.6]{Pol} by taking a faithfully flat extension of the local ring that is Noetherian and $F$-finite.
\end{proof}

\section{\'Etale-local extensions}\label{sec:fet}

We use the following trick to show that \'etale-local extensions preserve our invariant $s^e_{I_0}$,
Corollary \ref{cor:etaleSe}.

\begin{Lem}\label{lem:MinSadditive}
Given $R,I_0,e$ as in Definition \ref{def:truncatedSrel}, let $I_1,I_2$ be two ideals that satisfy
$\fm I_1\subseteq I_0\subsetneq I_1$ and $\fm I_2\subseteq I_0\subsetneq I_2$.
If $\lambda^e_{I_0}(R,I_1)=\lambda^e_{I_0}(R,I_2)=s^e_{I_0}(R)$,
then $\lambda^e_{I_0}(R,I_1+I_2)=s^e_{I_0}(R)$.
\end{Lem}
\begin{proof}
Let $a=p^{e\dim R}s^e_{I_0}(R)$. 
Then for all ideals $I$ that contains $I_0$ and satisfies $\fm I\subseteq I_0$,
we have $al(I/I_0)\leq l(\ppower{I}{e}/\ppower{I_0}{e})$,
with equality if and only if $I_0=I$ or $\lambda^e_{I_0}(R,I)=s^e_{I_0}(R)$.
Thus
\begin{align*}
a\left(l(I_1/I_0)+l(I_2/I_0)\right)
    &=l(\ppower{I_1}{e}/\ppower{I_0}{e})+l(\ppower{I_2}{e}/\ppower{I_0}{e})\\
    &=l\left(\frac{\ppower{I_1}{e}+\ppower{I_2}{e}}{\ppower{I_0}{e}}\right)+l\left(\frac{\ppower{I_1}{e}\cap\ppower{I_2}{e}}{\ppower{I_0}{e}}\right)\\
    &\geq l\left(\frac{\ppower{(I_1+I_2)}{e}}{\ppower{I_0}{e}}\right)+l\left(\frac{\ppower{(I_1 \cap I_2)}{e}}{\ppower{I_0}{e}}\right)\\
    &\geq
    a\left(l\left(\frac{I_1+I_2}{I_0}\right)+l\left(\frac{I_1 \cap I_2}{I_0}\right)\right)\\
\end{align*}
Now equality holds everywhere. Since $I_1+I_2\neq I_0$ we see $\lambda^e_{I_0}(R,I_1+I_2)$ is well-defined and equal to $s^e_{I_0}(R)$.
\end{proof}

\begin{Cor}\label{cor:etaleSe}
Let $(R,\fm)\to (S,\fn)$ be an \'etale-local map of Noetherian local rings and let $I_0$ be $\fm$-primary.
Then for all $e$, $s^e_{I_0}(R)=s^e_{I_0S}(S)$.
\end{Cor}
\begin{proof}
We may replace $R$ and $S$ by $R/\ppower{I_0}{e}$ and $S/\ppower{(I_0S)}{e}$ to assume $R$ and $S$ Artinian,
so $R\to S$ is finite \'etale.
There is a compatible choice of coefficient fields of $R$ and $S$ (see for example \cite[(28.J) Theorem 60]{Matsumura}),
and thus
$S\cong R\otimes_k l$ where $l/k$ is finite separable and $k$ is a coefficient field of $R$.
By Lemma  \ref{lem:FlatLengthPropotional} we may enlarge $l$ and assume $l/k$ finite Galois with group $G$,
so $G$ acts on $S=R\otimes_k l$ on the second factor.

Let $J$ be an ideal of $S$ such that $\fm J \subseteq I_0S\subsetneq J$ and that $\lambda^e_{I_0S}(S,J)=s^e_{I_0S}(S)$.
Then $\sum_{g\in G} g(J)$ satisfies the same by Lemma \ref{lem:MinSadditive} and is of the form $IS$ where $I$ is an ideal of $R$ (cf. \citestacks{0CDR}).
By flatness, $\fm I\subseteq I_0\subsetneq I$.
Our equality now follows from Lemma  \ref{lem:FlatLengthPropotional}.
\end{proof}

\section{\texorpdfstring{$\Gamma$}{Gamma}-construction and the truncated invariants}\label{sec:TruncatedGamma}

The goal of this section is to prove the following.
Consequences will be discussed in Section \ref{sec:mainThms}.

\begin{Thm}\label{thm:GammaSe}
Let $A$ be a Noetherian complete local ring, $k$ a coefficient field of $A$, and $\Lambda$ a $p$-basis of $k$.
Let $R$ be a finite type $A$-algebra.
Let $Q\in\Spec(R)$ and let $I_0$ be a $QR_Q$-primary ideal of $R_Q$. 
Let $e$  be a positive integer.

Then there exists a cofinite subset $\Gamma_0=\Gamma_0(Q,I_0,e)$ of
$\Lambda$ such that for all cofinite subsets $\Gamma$ of $\Gamma_0$,
$s^{e}_{I_0}(R_Q)=s^{e}_{I_0R^\Gamma_{Q^\Gamma}}(R^\Gamma_{Q^\Gamma})$.
\end{Thm}

Before we go into the argument, let us discuss the idea of the proof.
To give an ideal $I$ with $\fm I\subseteq I_0\subsetneq I$
is the same as to give (an equivalence class of) a non-zero $n$-by-$n$ matrix if we fix a basis of the $n$-dimensional $\kappa(Q)$-vector space $(I_0:_{R_Q}Q)/I_0$.
The corresponding invariant is then ``algebraically computed''
by the entries of the matrix.
The set where the invariant takes a certain value should then be constructible, 
and if a constructible subset of the variety of matrices has points over small $\kappa(Q^\Gamma)$, then it ``almost'' has a $\kappa(Q)$-point, see Lemma  \ref{lem:AllGammaEpoint}.
However, 
we run into the problem that $R_Q/\ppower{I_0}{e}$ and $R^\Gamma_{Q^\Gamma}/\ppower{(I_0R^\Gamma_{Q^\Gamma})}{e}$ may not have compatible coefficient fields.
We remedy this by choosing and fixing a coefficient field in a ring $B$ that dominates all those we shall consider,
and showing a weaker constructibility result (Lemma  \ref{lem:SeDefinable}).\\

Now we begin to prove the theorem.
We fix $A,k,\Lambda,R,Q,I_0,e$ as in the theorem. Let $d=\HT Q$. 

\subsection{Parametrization of socle ideals}\label{subsec:parameterizeIdeals}
Choose and fix $\epsilon_1,\ldots,\epsilon_n\in (I_0:_{R_Q} Q)$ that map to a basis of $(I_0:_{R_Q} Q)/I_0$.

Let $R_Q\to D$ be a flat map of local rings such that $QD$ the maximal ideal of $D$.
Then
$\epsilon_1,\ldots,\epsilon_n$ are in $(I_0D:_{D} Q)$ and map to a basis of $(I_0D:_{D} Q)/I_0D$.
Let $a_{ij}\ (1\leq i,j\leq n)$ be elements of the residue field $D/QD$,
and let $\widetilde{a_{ij}}\in D$ be arbitrary lifts.
Then the ideal
\[
J=J(D;a_{ij}\ (1\leq i,j\leq n)):=I_0D+\left(\sum_j \widetilde{a_{ij}}\epsilon_j\mid 1\leq i\leq n\right)
\]
is independent of the choice of $\widetilde{a_{ij}}$ and $QJ\subseteq I_0\subseteq J$.
We have $I_0\neq J$ if and only if not all $a_{ij}$ are zero.
Therefore we get a map
\[
J(D,-)\colon(\bA^{n^2}_{\kappa(Q)}\setminus \{0\})(D/QD)\to \{J\subseteq D\mid QJ\subseteq I_0D\subsetneq J\}.
\]
This map is surjective since $\epsilon_1,\ldots,\epsilon_n$ map to a basis of $(I_0D:_{D} Q)/I_0D$.

If $D\to D'$ is another flat map of local rings with $QD'$ the maximal ideal of $D'$,
then we have a commutative diagram
\[
\begin{CD}
(\bA^{n^2}_{\kappa(Q)}\setminus \{0\})(D/QD)@>{J(D;-)}>> \{J\subseteq D\mid QJ\subseteq I_0D\subsetneq J\}\\
@VVV @VV{J\mapsto JD'}V\\
(\bA^{n^2}_{\kappa(Q)}\setminus \{0\})(D'/QD')@>{J(D';-)}>> \{J\subseteq D'\mid QJ\subseteq I_0D'\subsetneq J\}
\end{CD}
\]

\subsection{Expressing the length using the entries of the matrix}
If we choose a coefficient field of $R_Q/\ppower{I_0}{e}$
and then a basis of this ring as a vector space over the chosen coefficient field,
then we can write out elements explicitly and calculate the length
$l(\ppower{J}{e}/\ppower{I_0}{e})$
as a constructible function of $a_{ij}$, where $J=J(R_Q;a_{ij})$, as explained below.
The same can be done for each $R^\Gamma_{Q^\Gamma}$.
However, if the choice of coefficient fields is not compatible,
then the constructible functions are also not compatible.
We get around this issue as follows.

Choose and fix a cofinite subset $\Gamma^*$ of $\Lambda$ such that for all $\Gamma\subseteq \Gamma^*$,
$QR^\Gamma$ is prime (Lemma  \ref{lem:GammaFacts}(\ref{QRGammaPrime})).
%
Let $B$ be the strict Henselization of $R^{\Gamma^*}_{Q^{\Gamma^*}}.$
Then $(B,QB)$ is a Noetherian Henselian local ring with $k_B:=B/QB$ separably closed.
Note that $k_B$ is 
a separable closure of $\kappa(Q^{\Gamma^*})$,
thus algebraic over $\kappa(Q)$ by Lemma \ref{lem:GammaFacts}(\ref{PiGammaUnivHomeo})
and generalities of universal homeomorphisms, see for example \citestacks{01S4}.


Choose and fix a coefficient field $\sigma\colon k_B\to B/\ppower{(I_0B)}{e}$.
We remind the reader again that $k_B$ does not necessarily contain a coefficient field of $R/\ppower{I_0}{e}$.
Choose and fix a basis $\epsilon'_1,\ldots,\epsilon_m'$ of $B/\ppower{(I_0B)}{e}$
over this coefficient field.
We may therefore write
\[
\epsilon_i^{p^e}=\sum_{k} \sigma(c_{ik})\epsilon'_k
\]
where $c_{ik}\in k_B$.
If we change the choice of $\sigma$ and $\underline{\epsilon'}$ then $c_{ik}$ may change.
Next, write
\[
\epsilon'_k\epsilon'_l=\sum_{h} \sigma(d_{klh})\epsilon'_h.
\]
Again, the choice of $\sigma$ and $\underline{\epsilon'}$ may affect $d_{klh}$.

For $a_1,\ldots,a_n\in k_B$,
and any lifts $\widetilde{a_i}\in B$,
one has
\[
\left(\sum_i \widetilde{a_i}\epsilon_i\right)^{p^e}=\sum_{k} \sigma\left(\sum_i a_i^{p^e} c_{ik}\right)\epsilon'_k\in B/\ppower{(I_0B)}{e}
\]
This is because $\widetilde{a_i}-\sigma(a_i)\in QB$,
so $\left(\widetilde{a_i}-\sigma(a_i)\right)\epsilon_i\in I_0B$
and
$\widetilde{a_i}^{p^e}\epsilon_i^{p^e}-\sigma(a_i)^{p^e}\epsilon_i^{p^e}\in \ppower{(I_0B)}{e}$.

Now let $M=(a_{ij})$ be a nonzero $n$-by-$n$ $k_B$-matrix,
and consider the ideal $J=J(B;a_{ij})$ as in (\ref{subsec:parameterizeIdeals}).
$J$ is generated by $I_0B$ and $\sum_j \widetilde{a_{ij}}\epsilon_j\ (1\leq i\leq n)$,
where $\widetilde{a_{ij}}\in B$ are arbitrary lifts of $a_{ij}\in k_B$.
Therefore the ideal $\ppower{J}{e}/\ppower{(I_0B)}{e}$ of $B/\ppower{(I_0B)}{e}$
is spanned, as a $\sigma(k_B)$-vector space,
by the elements
\[\left(\sum_j\widetilde{a_{ij}}\epsilon_j\right)^{p^e}\epsilon'_l\]
where $1\leq i\leq n,1\leq l\leq m$.
By previous calculations,
\begin{align*}
\left(\sum_{j=1}^n\widetilde{a_{ij}}\epsilon_j\right)^{p^e}\epsilon'_l&=\sum_{k=1}^m \sigma\left(\sum_{j=1}^n a_{ij}^{p^e} c_{jk}\right)\epsilon'_k\epsilon'_l\\
&=\sum_{h=1}^m \sigma\left(\sum_{k=1}^m\sum_{j=1}^n a_{ij}^{p^e} c_{jk}d_{klh}\right)\epsilon'_h
\end{align*}

Let
$M'$ be the $m$-by-$mn$ matrix whose $(n(l-1)+i)$th column is the vector
$\left(\sum_{k=1}^m\sum_{j=1}^n a_{ij}^{p^e} c_{jk}d_{klh}\right)_{1\leq h\leq m}$.
Consider the function $f\colon M\mapsto \frac{\rank M'}{p^{ed}\rank M}$.
We have $f(M)=\lambda^e_{I_0B}(B,J)$.

\subsection{Constructibility}
We now know that for $M\in (\bA^{n^2}_{\kappa(Q)}\setminus \{0\})(k_B)$,
\[\lambda^e_{I_0B}(B,J(B;M))=\frac{\rank M'}{p^{ed}\rank M}=f(M).\]
Note that the rank of a matrix is detected by the (non-)vanishing of minors, 
and minors are polynomials of the entries of the matrix.
The problem now is that the entries of $M'$, being polynomials of the entries of $M$ with $k_B$-coefficients, 
are not necessarily polynomials with $\kappa(Q)$-coefficients.
We have the following remedy.

\begin{Lem}\label{lem:SeDefinable}
%
Let $L=\kappa(Q)[c_{jk},d_{klh}]$, a finite extension of $\kappa(Q)$.

For each $s\in \bQ$, there is a locally closed subset $\cC_s\subseteq \bA^{n^2}_{\kappa(Q)}\setminus \{0\}$ with the following properties.

\begin{enumerate}
    \item \label{SeLessIfinC} If $M=(a_{ij})\in \cC_s(k_B)$, then $f(M)\leq s$; and
    \item \label{SeDefinedByC} If $M=(a_{ij})\in (\bA^{n^2}_{\kappa(Q)}\setminus \{0\})(k')$, where $k'$ is a subextension of $k_B/\kappa(Q)$ linearly disjoint with $L$ over ${\kappa(Q)}$, then $f(M)=s$ if and only if $M\in \cC_s(k')$.
\end{enumerate}
\end{Lem}
\begin{proof}
We can stratify $\bA^{n^2}_{\kappa(Q)}\setminus\{0\}$ according to rank, so 
it suffices to prove the statement for the function $g(M)=\rank M'$.
We may assume $s\in \bZ$, $0\leq s\leq m$, otherwise we can take $\cC_s=\emptyset$.

Fix $\bZ$-polynomials $G_1,\ldots,G_u,G_{u+1},\ldots,G_{u+v}$ of $n^2+mn+m^3$ variables such that $G_1(a_{ij},c_{jk},d_{klh}),\ldots,G_u(a_{ij},c_{jk},d_{klh})$ are the $s$ by $s$ minors of $M'$ and that $G_{u+1}(a_{ij},c_{jk},d_{klh}),\ldots,G_{u+v}(a_{ij},c_{jk},d_{klh})$ are the $s+1$ by $s+1$ minors of $M'$.
Here $u=\binom{m}{s}\binom{mn}{s}$ and $v=\binom{m}{s+1}\binom{mn}{s+1}$.

Let $\lambda_1,\ldots,\lambda_w$ be a basis of $L$ over $\kappa(Q)$.
Then for each $r\leq u+v$ we can find $k$-polynomials $G_{r1},\ldots,G_{rw}$ of $n^2$ variables such that
\[
G_r(X_{ij},c_{jl},d_{klh})=\lambda_1G_{r1}(X_{ij})+\ldots+ \lambda_wG_{rw}(X_{ij}),
\]
by writing every monomial of $c_{jk},d_{klh}$ as a $k$-linear combination of $\lambda_1,\ldots,\lambda_w$.

Finally, set
\[
\cC_s=\left(\bigcup_{1\leq r\leq u,1\leq p\leq w}D(G_{rp})\right)\cap\bigcap_{u+1\leq r\leq u+v,1\leq p\leq w}V(G_{rp}).
\]
Here, $V(H)$ is the vanishing set of a polynomial $H$ and $D(H)$ is the complement of $V(H)$.
For $M=(a_{ij})\in\cC_s(k_B)$, all $(s+1)$-by-$(s+1)$ minors of $M'$ are zero since $G_{rp}(M)=0$ for all $u+1\leq r\leq u+v$ and all $p$, so $\rank M'\leq s$.
If $k'$ is linearly disjoint with $L$ over $\kappa(Q)$, then $\lambda_1,\ldots,\lambda_w$ are linearly independent over $k'$,
so $M\in \cC_s(k')$ if and only if all $(s+1)$-by-$(s+1)$ minors of $M'$ are zero and some $s$-by-$s$ minor of $M'$ is nonzero, that is, $\rank M'=s$.
\end{proof}

\subsection{Conclusion of the proof}
Choose and fix a cofinite subset $\Gamma^\dagger$ of
$\Gamma^*$ such that $\kappa(Q^\Gamma)$ is linearly disjoint with the field $L$ in Lemma \ref{lem:SeDefinable} 
over $\kappa(Q)$ for all $\Gamma\subseteq \Gamma^\dagger$
(Lemma  \ref{lem:GammaFacts}(\ref{FoverKQLinearDisjt})). 
We now take $\Gamma_0$ to be any cofinite subset of $\Gamma^\dagger$ such that for all $\Gamma\subseteq\Gamma_0$,
the finite sets $\cS_{I_0R^\Gamma_{Q^\Gamma}}^{e}(R^\Gamma_{Q^\Gamma})\subseteq [0,l(R_Q/\ppower{I_0}{e})]\cap \frac{1}{p^{ed}l(R_Q/I_0)!}\bZ$
are all the same, which we denote by $\cS$.
This is possible because for $\Gamma\subseteq\Gamma'\subseteq\Gamma^\dagger$,
$\cS_{I_0R^\Gamma_{Q^\Gamma}}^{e}(R^\Gamma_{Q^\Gamma})\subseteq \cS_{I_0R^{\Gamma'}_{Q^{\Gamma'}}}^{e}(R^{\Gamma'}_{Q^{\Gamma'}})$
by Lemma \ref{lem:FlatLengthPropotional}.
Let $s=\min \cS$,
so $s=s^e_{I_0R^{\Gamma}_{Q^{\Gamma}}}(R^{\Gamma}_{Q^{\Gamma}})$
for all $\Gamma\subseteq\Gamma_0$.
We will show $s^e_{I_0}(R_Q)=s$,
finishing the proof.
Note that $s^e_{I_0}(R_Q)\geq s$ by Lemma \ref{lem:FlatLengthPropotional},
so it suffices to show $s^e_{I_0}(R_Q)\leq s$.

Let $\cC_s$ be as in Lemma \ref{lem:SeDefinable}.
For each $\Gamma\subseteq\Gamma_0$,
since $s\in \cS$,
there exists an ideal $J$ of $R^\Gamma_{Q^\Gamma}$
such that $QJ\subseteq I_0R^\Gamma_{Q^\Gamma}\subsetneq J$
and that $\lambda_{I_0R^\Gamma_{Q^\Gamma}}^{e}(R^\Gamma_{Q^\Gamma},J)=s$.
$J$ is of the form $J(R^\Gamma_{Q^\Gamma},M)$ 
for some $M\in (\bA^{n^2}_{\kappa(Q)}\setminus \{0\})(\kappa(Q^\Gamma))$,
see (\ref{subsec:parameterizeIdeals}).
$M$ 
is then in $\cC_s(\kappa(Q^\Gamma))$ by Lemma \ref{lem:FlatLengthPropotional}
and Lemma \ref{lem:SeDefinable}(\ref{CoeffContainSmallGamma}).
Thus we have $\cC_s(\kappa(Q^\Gamma))\neq\emptyset$ for all $\Gamma\subseteq\Gamma_0$.
Lemma \ref{lem:AllGammaEpoint} applies by Lemma  \ref{lem:GammaFacts}(\ref{FoverKQLinearDisjt}), so we can find a finite separable extension $E/\kappa(Q)$ such that $\cC_s(E)\neq\emptyset$.

Fix an embedding $E\to k_B$, possible as $k_B$ separably closed.
Let $R_Q\to S$ be an arbitrary \'etale-local map whose closed fiber is $\Spec(E)$.
Then since $B$ is Henselian, $S$ maps uniquely into $B$
compatible with the chosen $E\to k_B$.
This makes $B$ a local $S$-algebra,
and $S\to B$ is flat by for example \citestacks{00MK}.
A point $M\in \cC_s(E)$ gives rise of an ideal $J=J(S;M)$ with $Q
J\subseteq I_0S\subsetneq J$,
see (\ref{subsec:parameterizeIdeals}).
By Lemma \ref{lem:FlatLengthPropotional} $f(M)=\lambda^{e}_{I_0S}(S,J)$.
By Lemma  \ref{lem:SeDefinable}(\ref{SeLessIfinC}) we have $\lambda^{e}_{I_0S}(S,J)\leq s$, so $s^{e}_{I_0S}(S)\leq s$.
By Corollary \ref{cor:etaleSe}, $s^{e}_{I_0}(R_Q)\leq s$,
as desired.

\section{Approximation of regular maps}\label{sec:GammaReg}

In this section we use differential module calculations and some diagram chasing 
to prove a result about ``approximating'' a regular map,
Theorem \ref{thm:GammaSmooth}.
Its consequence, Corollary \ref{cor:RegAGammaToBGamma},
is a useful d\'evissage tool for our Theorem \ref{thm:RegExtn}, and may be such in other situations.
The proofs of Theorems \ref{thm:GammaSrel} and \ref{thm:SemiContOnGring} 
do not directly depend on the materials in this section.

\begin{Lem}\label{lem:GammaLostWhatOmega}
Let $(A,\fm)$ be a Noetherian complete local ring, $k$ a coefficient field of $A$.
Let $\Lambda$ be a $p$-basis of $k$ and $\Gamma\subseteq\Lambda$ be a cofinite subset.

Then the kernel of the canonical map
$\Omega_{A/\bF_p}\otimes_A k^\Gamma\to \Omega_{A^\Gamma/\bF_p}\otimes_{A^\Gamma} k^\Gamma$
is spanned by $\mathrm{d}\lambda\ (\lambda\in\Gamma).$
\end{Lem}
\begin{proof}
Let $A_\Gamma=A\otimes_k k^\Gamma$.
The ring $A_\Gamma$ is Noetherian; we do not use this fact.
By definition, there is a canonical map $A_\Gamma\to A^\Gamma$ that is an isomorphism modulo $\fm^2$.
Examine the following commutative diagram of $k^\Gamma$-vector spaces with exact rows (cf. \citetwostacks{00S2}{07BP})
\[
\begin{CD}
H_1(L_{k^\Gamma/\bF_p})@>>>
\fm A_\Gamma/\fm^2 A_\Gamma @>>>
\Omega_{A_\Gamma/\bF_p}/\fm @>>>
\Omega_{k^\Gamma/\bF_p} @>>> 0\\
@V{\wr}VV @V{\wr}VV @VVV @V{\wr}VV @.\\
H_1(L_{k^\Gamma/\bF_p})@>>>
\fm A^\Gamma/\fm^2 A^\Gamma @>>>
\Omega_{A^\Gamma/\bF_p}/\fm @>>>
\Omega_{k^\Gamma/\bF_p} @>>> 0
\end{CD}
\]
we see $\Omega_{A_\Gamma/\bF_p}/\fm=\Omega_{A^\Gamma/\bF_p}/\fm$.
Thus our desired kernel is the kernel of $\Omega_{A/\bF_p}\otimes_A k^\Gamma\to \Omega_{A_\Gamma/\bF_p}/\fm$.
Next, consider the exact sequence \citestacks{00S2}
\[
\begin{CD}
H_1(L_{A_\Gamma/A})@>{\delta}>> \Omega_{A/\bF_p}\otimes_A A^\Gamma @>>>
\Omega_{A_\Gamma/\bF_p} @>>>
\Omega_{A_\Gamma/A} @>>> 0
\end{CD}
\]
We know that $A_\Gamma=A\otimes_k k^\Gamma=\bigcup_e A[X_{\lambda,e}\ (\lambda\in\Gamma)]/(X_{\lambda,e}^{p^e}-\lambda)$,
so $\Omega_{A_\Gamma/A}=0$.
Thus the kernel of $\Omega_{A/\bF_p}\otimes_A k^\Gamma\to \Omega_{A_\Gamma/\bF_p}/\fm$
is the image of the image of $\delta$.
Using the presentation once again
we see that this vector space is spanned by $\mathrm{d}\lambda\ (\lambda\in\Gamma).$
\end{proof}

\begin{Thm}\label{thm:GammaSmooth}
Let $\varphi\colon (A,\fm,k)\to (B,\fn,l)$ be a flat local map of Noetherian complete local rings such that $\overline{B}:=B/\fm B$ is geometrically regular over $k$.
Fix a coefficient field $k\subseteq A$ and a $p$-basis $\Lambda$ of $k$.
Then the following hold.

\begin{enumerate}
    \item\label{SmallGammaIndepInL} There exists a cofinite subset $\Gamma_0\subseteq \Lambda$ such that the image of $\Gamma_0$ in $l$ is $p$-independent.
    
    \item\label{CoeffContainSmallGamma} Let $\Gamma_0$ be as in $(\ref{SmallGammaIndepInL})$.
    Then there exists a coefficient field $l'\subseteq B$ that contains the set $\varphi(\Gamma_0)\subseteq B$.
    
    \item\label{AGammaInBGamma} Let $\Gamma_0,l'$ be as in $(\ref{CoeffContainSmallGamma})$,
    and extend $\varphi(\Gamma_0)$ to a $p$-basis $\Lambda'$ of $l'$.
    Then for each cofinite subset $\Gamma$ of $\Gamma_0$ and each cofinite subset $\Gamma'$ of $\Lambda'$
    with $\varphi(\Gamma)\subseteq \Gamma'$,
    there exists a canonical $A$-algebra map
    $A^{\Gamma}\to B^{\Gamma'}$
    that is flat and local.
    
    \item\label{SmallGammaSmooth} Let $\Gamma_0,l',\Lambda'$ be as in $(\ref{AGammaInBGamma})$.
    Then for each cofinite subset $\Gamma$ of $\Gamma_0$ there exists a cofinite subset $\Gamma_0'$ of $\Lambda'$ containing $\varphi(\Gamma)$,
    such that for any cofinite subset $\Gamma'$ of $\Gamma_0'$ containing $\varphi(\Gamma)$,
    the canonical map $A^{\Gamma}\to B^{\Gamma'}$ in $(\ref{AGammaInBGamma})$ is regular.
\end{enumerate}
\end{Thm}
\begin{proof}
We shall use, without further mentioning, that $p$-independence is the same as differential independence, and a $p$-basis is the same as a differential basis; see \citestacks{07P2}.

By \citestacks{07E5}, the canonical map
$\Omega_{k/\bF_p}\otimes_k l\to \Omega_{\overline{B}/\bF_p}\otimes_{\overline{B}} l$ is injective.
Let $V$ denote its image.
By the exact sequence \citestacks{00RU}
we see that $W:=\ker(\Omega_{\overline{B}/\bF_p}\otimes_{\overline{B}} l\to \Omega_{l/\bF_p})$
is a finite dimensional $l$-vector space.
Thus there exists a finite subset $\Lambda_0$ of $\Lambda$
such that $V\cap W$ is contained in the linear span of the images of $\mathrm{d}\lambda\ (\lambda\in\Lambda_0)$ in $V$.
Thus we see that for $\Gamma_0=\Lambda\setminus\Lambda_0$
the images of $\mathrm{d}\lambda\ (\lambda\in \Gamma_0)$
in $\Omega_{l/\bF_p}$ are $l$-linearly independent, showing (\ref{SmallGammaIndepInL}).

Let $\Gamma_0$ be as in (\ref{SmallGammaIndepInL}).
Note that $\Omega_{\bF_p(\Gamma_0)/\bF_p}$
has a basis $\mathrm{d}\lambda\ (\lambda\in \Gamma_0)$ and they are mapped to $l$-linearly independent elements of $\Omega_{l/\bF_p}$.
By \citestacks{07EL},
we see that $\bF_p(\Gamma_0)\to l$ is formally smooth,
thus by definition \citestacks{00TI} and the completeness of $B$
we can lift the identity map $l=B/\fn$ to an $\bF_p(\Gamma_0)$-algebra map $l\to B$,
showing (\ref{CoeffContainSmallGamma}).

Now let $\Gamma_0,l'$ be as in (\ref{CoeffContainSmallGamma})
and take $\Gamma\subseteq\Gamma_0$.
For any cofinite subset $\Gamma'$ of $\Lambda'$ containing $\varphi(\Gamma)$ and any positive integer $e$
we have a canonical map $k^{\Gamma,e}\to l^{\Gamma',e}$ of $k$-algebras, see Construction \ref{constr:Gamma}.
This induces a canonical map $A^\Gamma\to B^{\Gamma'}$ of $A$-algebras.
This map is clearly local, and it is flat since both $A^\Gamma$ and $B^{\Gamma'}$ are flat over $A$ (Lemma \ref{lem:GammaFacts}(\ref{GammaNoeFfin})) and $A^\Gamma/\fm A^\Gamma=k^\Gamma$ is a field, see \citestacks{00MK}.

Finally, let $\Gamma_0,l',\Lambda'$ be as in (\ref{AGammaInBGamma}) and fix a cofinite subset $\Gamma$ of $\Gamma_0$.
Then $A^\Gamma$ is excellent (Lemma \ref{lem:GammaFacts}(\ref{GammaNoeFfin})). 
Thus for any $\Gamma'$ containing $\varphi(\Gamma)$, the flat map $A^\Gamma\to B^{\Gamma'}$ is regular 
if $k^\Gamma\to \overline{B}^{\Gamma'}$ is formally smooth \cite{Andre}.
Note that $\overline{B}^{\Gamma'}$ is a regular local ring, since $\overline{B}$ is regular, the map $\overline{B}\to\overline{B}^{\Gamma'}$ is flat, and $\overline{B}^{\Gamma'}/\fn \overline{B}^{\Gamma'}=l^{\Gamma'}$ is a field.
By \citestacks{07E5},
we must show that the elements $\mathrm{d}\lambda\ (\lambda\in\Lambda\setminus\Gamma)$
are linearly independent in $\Omega_{\overline{B}^{\Gamma'}/\bF_p}/\fn$ for all small $\Gamma'$ containing $\varphi(\Gamma)$.

Recall that $\Omega_{k/\bF_p}\otimes_k l\to \Omega_{\overline{B}/\bF_p}\otimes_{\overline{B}} l$ is injective.
Therefore the images of the elements $\mathrm{d}\lambda\ (\lambda\in\Lambda\setminus\Gamma)$
are linearly independent in $\frac{\Omega_{\overline{B}/\bF_p}\otimes_{\overline{B}} l}{\operatorname{span}\{\mathrm{d}\lambda\mid\lambda\in \Gamma\}}$,
and by base change linearly independent in $\frac{\Omega_{\overline{B}/\bF_p}\otimes_{\overline{B}} l^{\Gamma'}}{\operatorname{span}\{\mathrm{d}\lambda\mid\lambda\in \Gamma\}}$.
Now, the ring map $\overline{B}\to \overline{B}^{\Gamma'}$ induces a commutative diagram 
\[
\begin{CD}
\Omega_{\overline{B}/\bF_p}\otimes_{\overline{B}} l^{\Gamma'} @>>>
\Omega_{l/\bF_p}\otimes_l l^{\Gamma'}\\ 
@V{\psi}VV @VVV\\ 
\Omega_{\overline{B}^{\Gamma'}/\bF_p}/\fn @>>>
\Omega_{l^{\Gamma'}/\bF_p}\\ 
\end{CD}
\]
and we know that $\ker(\psi)$ is spanned by $\mathrm{d}\lambda'\ (\lambda'\in\Gamma')$ by Lemma  \ref{lem:GammaLostWhatOmega}.
Note that $\varphi(\Gamma)$ is part of the $p$-basis $\Lambda'$,
and every element of $\Gamma$ becomes a $p$-power in $\overline{B}^{\Gamma'}$.
Thus if we factor out the images of all $\mathrm{d}\lambda\ (\lambda\in\Gamma)$ in the diagram,
we get a commutative diagram 
\[
\begin{CD}
\frac{\Omega_{\overline{B}/\bF_p}\otimes_{\overline{B}} l^{\Gamma'}}{\operatorname{span}\{\mathrm{d}\lambda\mid\lambda\in \Gamma\}} @>>>
\bigoplus_{\lambda'\in \Lambda'\setminus\varphi(\Gamma)} l^{\Gamma'}\mathrm{d}\lambda'\\ 
@V{\overline{\psi}}VV @V{\theta}VV\\ 
\Omega_{\overline{B}^{\Gamma'}/\bF_p}/\fn @>>>
\bigoplus_{\lambda'\in \Lambda'\setminus\Gamma'} l^{\Gamma'}\mathrm{d}\lambda'\\ 
\end{CD}
\]
where $\ker(\overline{\psi})$ is spanned by $\mathrm{d}\lambda'\ (\lambda'\in\Gamma'\setminus\varphi(\Gamma))$, thus maps isomorphically onto $\ker(\theta)$.
Since the elements $\mathrm{d}\lambda\ (\lambda\in\Lambda\setminus\Gamma)$
are linearly independent in $\frac{\Omega_{\overline{B}/\bF_p}\otimes_{\overline{B}} l^{\Gamma'}}{\operatorname{span}\{\mathrm{d}\lambda\mid\lambda\in \Gamma\}}$,
for them to be linearly independent in $\Omega_{\overline{B}^{\Gamma'}/\bF_p}/\fn$ it suffices that the space their images span in $\bigoplus_{\lambda'\in \Lambda'\setminus\varphi(\Gamma)} l^{\Gamma'}\mathrm{d}\lambda'$
has zero intersection with $\ker(\theta)$ (diagram chase).
Now the choice is clear: take a finite subset $\Lambda'_0$ of $\Lambda'\setminus\varphi(\Gamma)$
such that the images of $\mathrm{d}\lambda\ (\lambda\in\Lambda\setminus\Gamma)$ in $\Omega_{l/\bF_p}/\operatorname{span}\{\mathrm{d}\varphi(\lambda)\mid\lambda\in \Gamma\}=\bigoplus_{\lambda'\in \Lambda'\setminus\varphi(\Gamma)} l\mathrm{d}\lambda'$
are in $\bigoplus_{\lambda'\in \Lambda'_0} l\mathrm{d}\lambda'$,
and take $\Gamma'_0=\Lambda\setminus\Lambda'_0$.
\end{proof}

\begin{Cor}\label{cor:RegAGammaToBGamma}
Let $\varphi:(A,\fm)\to (B,\fn)$ be a flat local map of Noetherian complete local rings such that $\overline{B}:=B/\fm B$ is geometrically regular over $k$.

Fix a coefficient field $k\subseteq A$ and a $p$-basis $\Lambda$ of $k$.
Then there exists a coefficient field $l$ of $B$, not necessarily containing $k$, and a $p$-basis $\Lambda'$ of $l$ that satisfy the following:

For all cofinite subsets $\Gamma_1$ of $\Lambda$ and $\Gamma'_1$ of $\Lambda'$,
there exist cofinite subsets $\Gamma$ of $\Gamma_1$ and $\Gamma'$ of $\Gamma'_1$
and a local map $A^{\Gamma}\to B^{\Gamma'}$ of $A$-algebras that is regular.
\end{Cor}
\begin{proof}
Let $\Gamma_0$ a cofinite subset of $\Lambda$, $l=l'$ a coefficient field of $B$, and $\Lambda'$ a $p$-basis of $l$ be as in Theorem \ref{thm:GammaSmooth}(\ref{AGammaInBGamma});
we shall show that this choice of $l$ and $\Lambda'$ works.

Let $\Gamma$ be a cofinite subset of $\Gamma_0\cap\Gamma_1$ such that $\varphi(\Gamma)\subseteq \Gamma'_1$, and let $\Gamma'_0$ be as in Theorem \ref{thm:GammaSmooth}(\ref{SmallGammaSmooth}) for $\Gamma$.
Then we have $\varphi(\Gamma)\subseteq \Gamma'_0\cap \Gamma'_1$.
Thus by Theorem \ref{thm:GammaSmooth}(\ref{SmallGammaSmooth}) we see $\Gamma$ and $\Gamma':=\Gamma'_0\cap \Gamma'_1$ works.
\end{proof}

\section{Consequences}\label{sec:mainThms} 

We are now able to answer a few questions in \cite[\S 7]{ST}.
These results recover known results about $F$-rationality 
\cite[Theorems 2.2, 3.1, and 3.5]{Vel95}
via \cite[Proposition 2.1]{ST},
and can be viewed as quantitative versions of such.
Note that, as opposed to the case of $F$-signature,
in Theorem \ref{thm:GammaSrel},
we may not have equality before taking the supremum;
and in Theorem \ref{thm:RegExtn}, the closed fiber has to be assumed geometrically regular instead of just regular.
The same is true for the corresponding qualitative statements for $F$-rationality in \cite{Vel95}.
See \cite{QGSS23}.

\begin{Thm}[{\cite[Question 7.4]{ST}}]\label{thm:GammaSrel}
Let $A$ be a Noetherian complete local ring, $k$ a coefficient field of $A$.
Let $\Lambda$ be a $p$-basis of $k$.
Let $R$ be a finite type $A$-algebra and let $Q\in \Spec(R)$.

Then $s_\mathrm{rel}(R_Q)=\sup_\Gamma s_\mathrm{rel}(R^\Gamma_{Q^\Gamma})$,
where $\Gamma$ ranges over all cofinite subsets of $\Lambda$.
\end{Thm}
\begin{proof}
We know $s_\mathrm{rel}(R_Q)\geq\sup_\Gamma s_\mathrm{rel}(R^\Gamma_{Q^\Gamma})$ by \cite[Corollary 3.12]{ST}, and it suffices to show the reversed inequality.
We may assume $s_\mathrm{rel}(R_Q)>0$,
so $R_Q$ is $F$-rational
\cite[Proposition 2.1]{ST},
thus $R_Q$ is Cohen-Macaulay \cite[Theorem 4.2(c)]{HH94}.
By Lemma  \ref{lem:GammaFacts}(\ref{GammaNoeFfin}),
all $R^\Gamma_{Q^\Gamma}$ are Cohen-Macaulay,
hence Lemma \ref{lem:truncatedandnot} applies.

Let $I_0$ be a parameter ideal of $R_Q$, so $I_0R^\Gamma_{Q^\Gamma}$ is a parameter ideal of $R^\Gamma_{Q^\Gamma}$ for all $\Gamma$.
By Theorem \ref{thm:Polstra}(\ref{ExistsC}) (or \cite[Theorem 3.6]{Pol}) we can take a constant $C$ that controls $F$-colength differences for $R^{\Lambda}$ (Definition \ref{def:ConstantControl}).
By Theorem \ref{thm:Polstra}(\ref{CgoodforflatextnGLOBAL}) we see that $C$ controls $F$-colength differences for $R$ and all $R^\Gamma$ where $\Gamma$ is a cofinite subset of $\Lambda$.
Thus for all positive integers $e$, $|s_{I_0}^e(R_Q)-s_\mathrm{rel}(R_Q)|\leq Cp^{-e}$, and for all $\Gamma$,
$|s_{I_0R^\Gamma_{Q^\Gamma}}^e(R^\Gamma_{Q^\Gamma})-s_\mathrm{rel}(R^\Gamma_{Q^\Gamma})|\leq Cp^{-e}$.
From Theorem \ref{thm:GammaSe} we see
$\sup_\Gamma s_\mathrm{rel}(R^\Gamma_{Q^\Gamma})\geq s_\mathrm{rel}(R_Q)-2Cp^{-e}$ for all $e$, as desired.
\end{proof}

\begin{Thm}[{\cite[Question 7.3]{ST}}]\label{thm:RegExtn}
Let $R\to S$ be a flat local map of Noetherian local rings with geometrically regular closed fiber.
Then $s_\mathrm{rel}(R)=s_\mathrm{rel}(S)$.
\end{Thm}
\begin{proof}
Note that if $k$ is a field and $A$ is a Noetherian local geometrically regular $k$-algebra, then so is $A^\wedge$, since $A^\wedge\otimes_k k'$ is the completion of $A\otimes_k k'$ for all finite purely inseparable extensions $k'$ of $k$.
We may therefore assume $R,S$ complete
since completion does not change the relative $F$-rational signature (cf. \cite[Observation 2.2(1)]{HY}).
Fix a parameter ideal $I_0$ of $R$ and a parameter ideal $J_0$ of $S$.
We do not require $I_0R\subseteq J_0$.
Since $0\leq s_\mathrm{rel}(S)\leq s_\mathrm{rel}(R)$ \cite[Proposition 3.12]{ST},
we may assume $s_{\mathrm{rel}}(R)>0$,
so $R$ is $F$-rational
\cite[Proposition 2.1]{ST},
thus $R$ is Cohen-Macaulay \cite[Theorem 4.2(c)]{HH94}. 
Then $S$ is Cohen-Macaulay \citestacks{045J}.

Let $k$ be a coefficient field of $R$ and $\Lambda$ a $p$-basis of $k$.
Let $l$ be a coefficient field of $S$ and $\Lambda'$ a $p$-basis of $l$ as in Corollary \ref{cor:RegAGammaToBGamma}.
By Theorem \ref{thm:Polstra}(\ref{ExistsC}) (or \cite[Theorem 3.6]{Pol}) we can find a constant $C$ that controls $F$-colength differences for $S^{\Lambda'}$ (Definition \ref{def:ConstantControl}).

Fix a positive integer $e$. By Theorem \ref{thm:GammaSe},
there exists a cofinite subset
$\Gamma_1$ of $\Lambda$ such that such that for all cofinite subsets $\Gamma$ of $\Gamma_1$,
$s^e_{I_0}(R)=s^e_{I_0R^{\Gamma}}(R^{\Gamma})$;
and a cofinite subset $\Gamma'_1$ of $\Lambda'$,
such that for all cofinite subsets $\Gamma'$ of $\Gamma'_1$,
$s^e_{J_0}(S)=s^e_{J_0S^{\Gamma'}}(S^{\Gamma'})$.
Corollary \ref{cor:RegAGammaToBGamma} ensures that
we can find $\Gamma\subseteq \Gamma_1$ and $\Gamma'\subseteq \Gamma'_1$ that admit an $R$-algebra map $R':=R^\Gamma\to S^{\Gamma'}=:S'$ that is local and regular.
Since $R$ and $S$ are Cohen-Macaulay, so are $R'$ and $S'$, see Lemma  \ref{lem:GammaFacts}(\ref{GammaNoeFfin}).
Thus by \cite[Corollaries 5.9 and 5.14]{ST}, $s_{\mathrm{rel}}(R')=s_{\mathrm{rel}}(S')$.

By Theorem \ref{thm:Polstra}(\ref{CgoodforflatextnGLOBAL}) we see that $C$ controls $F$-colength differences for $R,R',S$, and $S'$.
In particular, we have $|s^e_{\fa}(D)-s_{\mathrm{rel}}(D)|\leq Cp^{-e}$ where $D$ is $R,R',S$, or $S'$ and $\fa$ is any parameter ideal of $D$ (Lemma  \ref{lem:truncatedandnot}).
By our choices we see that $s^e_{I_0}(R)=s^e_{I_0R'}(R')$ and
$s^e_{J_0}(S)=s^e_{J_0S'}(S')$.
Now, $s_{\mathrm{rel}}(R')=s_{\mathrm{rel}}(S')$ implies
$|s_{\mathrm{rel}}(R)-s_{\mathrm{rel}}(S)|\leq 4Cp^{-e}$.
Since $e$ was arbitrary,
we see $s_{\mathrm{rel}}(R)=s_{\mathrm{rel}}(S)$, as desired.
\end{proof}

A local ring $A$ is a \emph{G-ring} if $A$ is Noetherian and the fibers of $A\to A^\wedge$ are geometrically regular (cf. \citestacks{07PT}).

\begin{Thm}[{\cite[Question 7.1]{ST}}]\label{thm:SemiContOnGring}
Let $R$ be a finite type algebra over a local G-ring.
Then the function $\fp\mapsto s_\mathrm{rel}(R_\fp)$
is lower semi-continuous on $\Spec(R)$.
\end{Thm}
\begin{proof}
By Theorem \ref{thm:RegExtn} (and \citestacks{02JY}) it suffices to prove the result
for a finite type algebra over a Noetherian complete local ring.
Since the supremum of a family of lower semi-continuous functions is lower semi-continuous,
by the $\Gamma$-construction and Theorem \ref{thm:GammaSrel} it suffices to prove the result for an $F$-finite Noetherian ring $R$. 
Note that $R$ is excellent by Kunz's theorem \cite[(42.B) Theorem 108]{Matsumura}.
In particular, the Cohen-Macaulay locus of $R$ is open \cite[Proposition 6.11.8]{EGA4_2}.

Note $s_\mathrm{rel}(D)>0$ for a Noetherian local ring $D$ implies $D^\wedge$ $F$-rational \cite[Proposition 2.1]{ST}, hence normal and Cohen-Macaulay \cite[Theorem 4.2(b)(c)]{HH94}.
We see that if $s_\mathrm{rel}(R_\fp)>0$ for some $\fp\in\Spec(R)$ then there exists an $f\in R\setminus{\fp}$ such that $R_f$ is a Cohen-Macaulay domain.
Then on $\Spec(R_f)$ the relative $F$-rational signature coincides with the dual $F$-signature, and is lower semi-continuous, see \cite[Corollary 5.9 and Theorem 5.10]{ST}.
Since $s_\mathrm{rel}$ is always non-negative, this shows the lower semi-continuity of our function.
\end{proof}

\section{The case of the \texorpdfstring{$F$}{F}-rational signature}\label{sec:Srat}

Our method can be used to study the $F$-rational signature defined by Hochster-Yao \cite{HY}.
To be precise, in the notations of Definition \ref{def:truncatedSrel},
let
\[
s^{e,1}_{I_0}(R)=\min \{\lambda^e_{I_0}(R,I_0+uR)\mid u\in R, (I_0:u)=\fm\}.
\]
Then if $R^\wedge$ is equidimensional, $I_0$ is a parameter ideal, and $C$ controls colength differences for $\fm\in\Spec(R)$,
we have 
\[\left| s_{I_0}^{e,1}(R) -s_\mathrm{rat}(R)
\right|\leq Cp^{-e}\]
by \cite[Theorem 2.5]{HY}.
Here, $s_\mathrm{rat}(R)$ is the $F$-rational signature of $R$, denoted by
$\operatorname{r}_R(R)$ in \cite{HY}.
One can then apply the methods in \S\ref{sec:TruncatedGamma} to study the truncated invariant $s_{I_0}^{e,1}(R)$: 
instead of parametrizing the ideal $J$ with a nonzero $n$-by-$n$ matrix,
we parametrize the element $u$ by a nonzero $1$-by-$n$ vector.

The problem here is that we do not know if \'etale-local extensions preserve this truncated invariant (or their limit $s_\mathrm{rat}(R)$).
So, we cannot get the full statement of Theorem \ref{thm:GammaSe} (or Theorem \ref{thm:GammaSrel}) in this case.
However, the following holds. 
\begin{Prop}\label{prop:GammaSe1}
Let $A$ be a Noetherian complete local ring, $k$ a coefficient field of $A$, and $\Lambda$ a $p$-basis of $k$.
Let $R$ be a finite type $A$-algebra.
Let $Q\in\Spec(R)$ and let $I_0$ be a $QR_Q$-primary ideal of $R_Q$. 
Let $e$  be a positive integer.

Assume that $\kappa(Q)$ is separably closed.\footnote{Since we are in characteristic $p$, this is equivalent to say that $k$ is separably closed and that $Q$ is a maximal ideal of $R$ lying above the maximal ideal of $A$.}
Then there exists a cofinite subset $\Gamma_0=\Gamma_0(Q,I_0,e)$ of
$\Lambda$ such that for every cofinite subset $\Gamma$ of $\Gamma_0$,
$s^{e,1}_{I_0}(R_Q)=s^{e,1}_{I_0R^\Gamma_{Q^\Gamma}}(R^\Gamma_{Q^\Gamma})$.
\end{Prop}

Since $s_{\mathrm{rat}}(R)>0$ implies $R$ Cohen-Macaulay 
(in fact it implies $F$-rationality of $R^\wedge$, \cite[Theorem 4.1]{HY}),
we get from the proof of Theorem \ref{thm:GammaSrel} the following.

\begin{Cor}\label{cor:GammaSrat}
Let $A$ be a Noetherian complete local ring, $k$ a coefficient field of $A$, and $\Lambda$ a $p$-basis of $k$.
Let $R$ be a finite type $A$-algebra and let $Q\in\Spec(R)$. 

Assume that $\kappa(Q)$ is separably closed.
Then $s_\mathrm{rat}(R_Q)=\sup_\Gamma s_\mathrm{rat}(R^\Gamma_{Q^\Gamma})$,
where $\Gamma$ ranges over all cofinite subsets of $\Lambda$.
\end{Cor}

However, even if one can remove the separable closedness condition,
one does not necessarily get versions of 
Theorems \ref{thm:RegExtn} and \ref{thm:SemiContOnGring} since the $F$-finite case is not known. 

\end{document}